\documentclass[11pt]{amsart}
\usepackage{amsmath, amssymb, amsbsy, amsfonts, amsthm, latexsym, amsopn, amstext, amsxtra, euscript, amscd, color, mathrsfs}
\usepackage[normalem]{ulem}
\usepackage{soul}
\usepackage{lscape}

\PassOptionsToPackage{hyphens}{url}\usepackage{hyperref}
 
 \usepackage[capbesideposition=outside,capbesidesep=quad]{floatrow}

\restylefloat{table}
\restylefloat{table}
            
\usepackage{multirow,caption}
            
\usepackage{amscd}
\usepackage{color,enumerate}

\newcommand{\RNum}[1]{\lowercase\expandafter{\romannumeral #1\relax}}

\usepackage[colorinlistoftodos,prependcaption,textsize=tiny]{todonotes}

\newtheorem{thm}{Theorem}[section]
\newtheorem{lem}[thm]{Lemma}

\newtheorem{rmk}[thm]{Remark}

\newtheorem{thm-con}[thm]{Theorem-Conjecture}
\numberwithin{equation}{section}

\theoremstyle{definition}
\newtheorem{defn}[thm]{Definition}

\newcommand{\F}{\mathbb F} 
\newcommand{\Z}{\mathbb Z}

\begin{document}
\title[Local permutation polynomials and their companions]{Local permutation polynomials and their companions}

\author[S.U. Hasan]{Sartaj Ul Hasan}
\address{Department of Mathematics, Indian Institute of Technology Jammu, Jammu 181221, India}
\email{sartaj.hasan@iitjammu.ac.in}

\author[H. Kumar]{Hridesh Kumar}
\address{Department of Mathematics, Indian Institute of Technology Jammu, Jammu 181221, India}
\email{2021RMA2022@iitjammu.ac.in}
 \thanks{The first named author is partially supported by Core Research Grant CRG/2022/005418 from the Science and Engineering Research Board, Government of India. The second named author is supported by the Prime Minister’s Research Fellowship PMRF ID-3002900 at IIT Jammu.}

\begin{abstract}
Gutierrez and Urroz (2023) have proposed a family of local permutation polynomials over finite fields of arbitrary characteristic based on a class of symmetric subgroups without fixed points called $e$- Klenian groups. The polynomials within this family are referred to as $e$- Klenian polynomials. Furthermore, they have shown the existence of companions for the $e$- Klenian polynomials when the characteristic of the finite field is odd. Here, we present three new families of local permutation polynomials over finite fields of even characteristic. We also consider the problem of the existence of companions for the $e$- Klenian polynomials over finite fields of even characteristic. More precisely, we prove that over finite fields of even characteristic, the $0$- Klenian polynomials do not have any companions. However, for $e \geq 1$, we explicitly provide a companion for the $e$- Klenian polynomials. Moreover, we provide a companion for each of the new families of local permutation polynomials that we introduce.
\end{abstract}

\maketitle
 
\section{Introduction}

Denote, as usual, by $\F_{q}$ the finite field with $q$ elements and by $\F_q[X_1, \ldots,X_n]$ the ring of polynomials in $n$ variables $X_1, \ldots, X_n$ over $\F_q$, where $q$ is a power of a prime and $n$ is a positive integer. It is well-known~\cite{LNH_1997} that any map from $\F_q^n$ to $\F_q$ can be uniquely represented by a polynomial $f(X_1,X_2,\ldots,X_n) \in \F_q[X_1,X_2,\ldots,X_n]$ such that $\text{deg}_{X_i}(f)<q$ for all $i\in \{1,2,\ldots,n\}$, where $\text{deg}_{X_i}(f)$ denotes the degree of $f$ in the $i$-th variable $X_i$. In view of this, studying the permutation behaviour of a function from $\F_q^n$ to $\F_q$ is the same as studying the permutation behaviour of a polynomial in $n$ variables over $\F_q$ with degree less than $q$ in each variable. Thus, in what follows, instead of functions from $\F_q^n$ to $\F_q$, we shall always consider the polynomials in $n$ variables over $\F_q$ with degree less than $q$ in each variable. The degree of $f=\sum  a_{i_1,\ldots, i_n} X_1^{i_1} \cdots X_n^{i_n}$, denoted by $\deg(f)$, is defined as $\deg(f) := \max\{i_1+ \cdots +i_n \mid a_{i_1,\ldots, i_n} \neq 0\}.$ A polynomial $f\in\F_q[X_1,X_2,\ldots,X_n]$ is called a permutation polynomial (PP) if the equation $f(X_1,X_2,\ldots,X_n)=a$ has exactly $q^{n-1}$ solutions in $\F_q^{n}$ for each $a\in \F_q$. For $n=1$, the above definition coincides with the classical notion of the permutation polynomial in one variable. The classification of permutation polynomials in several variables is generally a difficult problem. In 1970, Niederreiter \cite{NH_1970} classified the permutation polynomials of several variables over $\F_q$ of degree at most two. However, the classification of permutation polynomials in several variables of degree $\geq 3$ remains an open problem.

A polynomial $f(X_1,X_2,\ldots,X_n) \in\F_q[X_1,X_2,\ldots,X_n]$ is called a local permutation polynomial (LPP) if for each $ i \in \{1, \ldots, n\}$, the equation $f(a_1,\ldots,a_{i-1},X_i,a_{i+1},\ldots,a_n)=a$ has at most one solution for all $a \in \F_q$ and $\overline{a}_i\in \F_q^{n-1}$, where $\overline{a}_i:=(a_1,\ldots,a_{i-1},a_{i+1},\ldots,a_n)$. In other words, $f(a_1,a_2,\ldots,a_{i-1},X_i,a_{i+1},\ldots,a_n)\in \F_q[X_i]$ is a permutation polynomial over $\F_q$ for each $\overline{a}_i\in \F_q^{n-1}$, where  $ i \in \{1,\ldots, n\}$. It is easy to observe that an LPP is always a permutation polynomial; however, the converse may not be true in general. For instances, $f(X_1,X_2)=X_1^2+X_2$ is a permutation polynomial, but it is not an LPP over $\F_q$ when $q$ is a power of an odd prime. LPPs were first studied by Mullen~\cite{Mullen_1980, Mullen_1_1980}, who provided necessary and sufficient conditions for polynomials in two and three variables to be LPPs over prime fields. In $2004$, Diestelkamp, Hartke and Kenney \cite{DH_2004} studied the degree of an LPP in two variables over the finite field $\F_q$ and showed that the degree is at most $2q-4$, and that this bound is sharp. Anbar, Ka\c{s}{\i}kc{\i} and Topuzo\u glu~\cite{AKT_2019} generalized the notion of LPPs and defined what is known as vectorial local permutations. A map $F: \F_q^n \rightarrow \F_q^n$ given by
\[
F(X_1,\ldots,X_n) = (f_1(X_1,\ldots,X_n), \ldots,f_n(X_1,\ldots,X_n))
\]
is called a vectorial local permutation if all the polynomials $f_i(X_1,\ldots,X_n)\in \F_q[X_1,X_2,\ldots,X_n]$, where $i \in \{1, \ldots, n\}$, are LPPs over $\F_q$, and $F$ is called a vectorial permutation if it induces a permutation of $\F_q^n$. If $F$ is a vectorial permutation, then by using the absolute irreducibility of certain curve over $\F_q$ and the number of rational points on it, the authors showed that the degrees of the components $f_1, f_2, \ldots, f_n \in \F_q[X_1, \ldots, X_n]$ are at least
 two when $2 \leq \deg(F )=d < \sqrt{q}$ and $d\mid (q-1)$, where $\deg(F ):=\max \{\deg (f_i) \mid i=1, \ldots, n\}$. 
 
The notion of LPPs is intimately related to what are called Latin squares, which we shall now describe. A Latin square of order $n$ is an $n \times n$ matrix with entries from a set $S$ of size $n$, such that every element of $S$ appears exactly once in each row and each column. While Latin squares have a long history, it was Euler who initiated their systematic study in the 18th century. The simplest examples of Latin squares are the multiplication tables (also referred to as Cayley tables) of groups. Two Latin squares are said to be orthogonal if, when superimposed, they produce all the ordered pairs of $S \times S$. A set of Latin squares of order $n$ that are pairwise orthogonal is referred to as Mutually Orthogonal Latin Squares (MOLS). A set of MOLS of order $n$ is called a complete set if it contains $n-1$ Latin squares. It is a well-known fact~\cite{Mullen_book} that there is a one-to-one correspondence between the set of Latin squares of order a prime power $q$ and the set of bivariate LPPs over the finite fields of order $q$. Likewise, an orthogonal Latin square of the Latin square associated with a particular bivariate LPP corresponds to an orthogonal LPP, also referred to as a companion, associated with that LPP. Consequently, constructing an LPP (or a companion) is essentially the same as constructing a Latin square (or an orthogonal Latin square). It is interesting to note that constructing MOLS is a challenging combinatorial problem; however, a companion of a given bivariate LPP provides an elegant algebraic method for constructing a complete set of MOLS.

Recently, Gutierrez and Urroz~\cite{JJ_2023} established a one-to-one correspondence between bivariate LPPs and a specific type of $q$-tuples of permutation polynomials in one variable over the finite fields with $q$ elements. They then used this correspondence to construct a family of bivariate LPPs based on a class of symmetric subgroups without fixed points, known as $e$- Klenian groups. We refer to the LPPs in this family as e- Klenian polynomials. Motivated by the seminal work of Gutierrez and Urroz~\cite{JJ_2023}, we construct three new families of bivariate LPPs  based on certain subgroups of the symmetric group of permutations on $q$ symbols. As alluded to earlier, two LPPs are said to be companions of each other if their corresponding Latin squares are orthogonal. Generally, finding a companion for a given LPP is a challenging task. In~\cite{JJ_2023}, the authors showed the existence of a companion for the $e$- Klenian polynomials over finite fields of odd characteristic. For even characteristic, the authors presented some computational results regarding the companions of the $e$- Klenian polynomials. Inspired by their computational results, we prove the existence of a companion for the $e$- Klenian polynomials over finite fields of even characteristic for all $e \geq 1$. However, for $e=0$, we show that $e$- Klenian polynomials over finite fields of even characteristic do not have any companion. We also provide an explicit expression of a companion for each of our families of LPPs.

The structure of the paper is as follows. In Section~\ref{S2}, we recall some definitions and lemmas that will be used in the sequel. Section~\ref{S3} presents the construction of three new families of bivariate LPPs and discusses their equivalence. We investigate in Section \ref{S4} the existence results concerning companions of $e$- Klenian polynomials over finite fields of even characteristic, as well as of each of the new families of LPPs introduced in this paper. Finally, we conclude the paper in Section~\ref{S5}.

\section{Preliminaries}\label{S2}
Let ${\mathfrak S}q$ denote the symmetric group of permutations on the $q$ elements of the finite field $\F_q$. We list the elements of $\F_q$ as $\F_{q} = \{c_0, c_1, \ldots, c_{q-1}\}$ and will use this enumeration throughout the paper. We call a $q$-tuple $(\beta_0,\ldots,\beta_{q-1})\in {\mathfrak S}_q^q$ a permutation polynomial tuple if it satisfies that for $i \neq j$, $\beta_i^{-1}\beta_j$ has no fixed point. It is easy to see that if the elements of the tuple $(\beta_0,\beta_1,\ldots,\beta_{q-1})\in {\mathfrak S}_q^q$ form a subgroup of ${\mathfrak S}_q$, then $(\beta_0,\beta_1,\ldots,\beta_{q-1})$ is a permutation polynomial tuple if and only if  no $\beta_i, 0\leq i\leq q-1$, has a fixed point, except the identity permutation. In~\cite{JJ_2023}, the authors showed that there is a one to one correspondence between the set of bivariate LPPs over $\F_q$ and the set of permutation polynomial tuples over $\F_q$. More precisely, the authors proved the following result.

\begin{lem}\cite[Lemma 7]{JJ_2023}
\label{tpl}
There is a bijective map between the set of local permutation polynomials $f\in \F_q[X,Y]$ and the set of permutation polynomial tuples $\underline{\beta}_f :=(\beta_0,\ldots,\beta_{q-1}) \in {\mathfrak S}_q^q$. Furthermore, $f$ and $\underline{\beta}_f$ are associated to each other by the following relation: for each $0 \leq i \leq q-1$, $f(x,\beta_i(x))=c_i$ for all $x\in \F_q$.
\end{lem}
If the elements of the permutation polynomial tuple $\underline{\beta}_f$ form a subgroup of ${\mathfrak S}_q$, then the corresponding LPP $f$ is referred to as a permutation group polynomial. 

In the study of bivariate LPPs, we classify them with respect to certain equivalence relations. Two LPPs $f,g \in \F_q[X,Y]$ are said to be equivalent if their corresponding permutation polynomial tuples 
$(\beta_0,\ldots,\beta_{q-1})$ and $(\gamma_0,\ldots,\gamma_{q-1}) \in {\mathfrak S}_q^q$ satisfy the following relation:
\[
\gamma_i= \sigma \circ \beta_i \circ \lambda,~\mbox{for all}~i=0,1, \ldots, q-1,
\]
where $\sigma, \lambda \in {\mathfrak S}_q$. 

To the best of our knowledge, the only known class of permutation group polynomials, referred to as the $e$- Klenian polynomials, is due to Gutierrez and Urroz~\cite{JJ_2023} and is described in the following lemma.
\begin{lem}\cite[Corollary 17]{JJ_2023}
\label{JJ_tpl}
Let $q=p^r$, where $p$ is a prime number and $r$ is a positive integer. Let $0 \leq e < r$, $\ell =p^e$ and  $t=\frac{q}{\ell}$. Let $\alpha=C_{0,\alpha} \cdots C_{t-1,\alpha}$, where
\[
C_{i,\alpha}=(c_{i\ell},c_{i\ell+1}, \ldots, c_{(i+1)\ell-1})~\mbox{for all}~0\leq i \leq t-1
\]
and $\beta=C_{0,\beta} \cdots C_{\ell-1,\beta}$, where
\[
C_{j,\beta}=(c_{j},c_{j+\ell}, \ldots, c_{j+(t-1)\ell})~\mbox{for all}~0\leq j \leq \ell-1.
\]
Then the set defined by $G=\{\alpha^{i}\beta^{j}: 0\leq i \leq \ell-1, 0\leq j \leq t-1 \}$ is a subgroup of ${\mathfrak S}_q$ such that each element of $G$, except the identity, has no fixed point and $|G|=q$. 
\end{lem} 
Two Latin squares $L_1=(a_{ij})_{0 \leq i,j \leq q-1}$ and $L_2=(b_{ij})_{0 \leq i,j \leq q-1}$ of order $q$ are said to be orthogonal if all the $q^2$ ordered pairs $(a_{ij},b_{ij}), 0 \leq i,j \leq q-1 $  are distinct. Let $f_1$ and  $f_2 \in \F_q[X,Y]$ be the LPPs corresponding to $L_1$ and $L_2$, respectively. Then $L_1$ and $L_2$ are orthogonal if and only if the following system of equations
\begin{equation*}
    \begin{cases}
        f_1(X,Y)&=a\\
        f_2(X,Y)&=b
    \end{cases}
\end{equation*}
has a unique solution $(x,y) \in \F_q \times \F_q$ for all $ (a,b) \in \F_q \times \F_q$. In this case, we call the LPPs $f_1$ and $f_2$ to be orthogonal (or companion) of each other~\cite{NH_1971}.  

The notion of a companion is not only defined for bivariate LPPs; it has also been defined for bivariate permutation polynomials~\cite{NH_1971}. For a given bivariate permutation polynomial over $\F_q$, there are exactly $\prod_{i=1}^q i^q$ companions, which are themselves permutation polynomials. However, the existence of a companion which is also an LPP for a given bivariate LPP is not guaranteed, as we shall see in Section~\ref{S4}. Note that, throughout this paper, by ``companion" of a given LPP, we always mean an LPP that is orthogonal to the given LPP.

For any positive integer $k\leq q$, a $k$-plex of a Latin square $L$ of order $q$ is a selection of $kq$ entries from $L$ such that exactly $k$ entries are chosen from each row and each column. Moreover, each symbol is chosen exactly $k$ times. A $1$-plex is known as transversal. It is well known~\cite{EulerT, KD_2015}  that if $n$ is even then the addition table for the integers modulo $n$, possesses no transversal.

\section{Permutation group polynomials}\label{S3}
One may recall that the bivariate LPP over $\F_q$ corresponding to a permutation polynomial tuple $(\beta_0,\ldots, \beta_{q-1}) \in {\mathfrak S}_q^q$ is called a permutation group polynomial if and only if the set $\{\beta_0,\ldots,\beta_{q-1}\}$ forms a subgroup of ${\mathfrak S}_q$. Thus, finding new permutation group polynomials is equivalent to finding permutation polynomial tuples $(\beta_0,\ldots, \beta_{q-1})$ such that the set $\{\beta_0,\ldots,\beta_{q-1}\}$ forms a subgroup of ${\mathfrak S}_q$. For the reader's convenience, boldface letters in the permutation cycles appearing in the following, as well as in the subsequent theorems are used to indicate that they do not follow any specific pattern. The theorem below introduces the first family of permutation group polynomials. 

\begin{thm}\label{T31}
Let $q=4k$ be a positive integer, where $k=2^{\ell}$ for some positive integer $\ell \geq 2$. Also, let
\[
a=(c_0,c_1)(c_2,c_3)\cdots(c_{4k-2},c_{4k-1}),
\]
and
\begin{equation*}
    \begin{split}
        b=& (c_0, c_{2k-2}, \ldots, c_{2m}, c_{2(2m+k-1)}, \ldots , c_{k-2}, c_{4k-6},c_k,c_{2k}, \ldots, c_{2n+k},c_{2(2n+k)}, \ldots, c_{2k-4}, c_{4k-8}, \\
        & {\bf c_{4k-1}, c_{4k-4}}) (c_1, c_{2k+1}, \ldots, c_{2m+1}, c_{2(2m+k)+1}, \ldots , c_{k-1}, c_{4k-3},c_{k+1},c_{2k-1}, \ldots, c_{2n+k+1},\\
        & c_{2(2n+k)-1}, \ldots, c_{2k-3}, c_{4k-9}, {\bf c_{4k-2}, c_{4k-5}}),
    \end{split}
\end{equation*}
where $m \in \{0,1,\ldots, \frac{k-2}{2}\}$ and $n \in \{0,1,\ldots, \frac{k-4}{2}\}$.
Then $G_1=\{b^{j}a^{i}: 0 \leq j \leq 2k-1=\frac{q}{2}-1, 0\leq i \leq 1\}$ is a subgroup of ${\mathfrak S}_q$ in which  no permutation has fixed points except the identity and $|G_1|=q$.
\end{thm}
\begin{proof}
Recall that $G_1$ is a subgroup of ${\mathfrak S}_q$ if and only if for any two elements $b^{j_1}a^{i_1}, b^{j_2}a^{i_2} \in G_1$, the element $b^{j_1}a^{i_1}(b^{j_2}a^{i_2})^{-1} \in G_1$. Notice that
\begin{equation*}
    \begin{split}
        b^{j_1}a^{i_1}(b^{j_2}a^{i_2})^{-1} &= b^{j_1}a^{i_1}(a^{i_2})^{-1}(b^{j_2})^{-1}\\
        &= b^{j_1}a^{i_1}(a^{-1})^{i_2}(b^{-1})^{j_2}\\
        &= b^{j_1}a^{i_1}a^{i_2}(b^{2k-1})^{j_2}~\quad \quad(\mbox{as}~a^{-1}=a)\\
        &= b^{j_1}a^{i_1+i_2}b^{(2k-1)j_2}.
    \end{split}
\end{equation*}
If $i_1=i_2$ then $a^{i_1+i_2}= I$, the identity element of ${\mathfrak S}_q$. Therefore $b^{j_1}a^{i_1}(b^{j_2}a^{i_2})^{-1} \in G_1$ in this case. Now, when $i_1 \neq i_2$ then we have two cases, either $(i_1, i_2)=(1,0)$ or $(i_1, i_2)=(0,1)$. In both the cases, we have $a^{i_1+i_2}= a$ and hence 
\[
 b^{j_1}a^{i_1}(b^{j_2}a^{i_2})^{-1}= b^{j_1}ab^{(2k-1)j_2}.
\]
In order to show that $b^{j_1}ab^{(2k-1)j_2}$ is of the form $b^{j_3}a^{i_3}$, for some $0 \leq j_3 \leq 2k-1=\frac{q}{2}-1, 0\leq i_3 \leq 1$, we shall use the identity $ab=b^{k+1}a$. As in this case
\[
b^{j_1}b^{(2k-1)j_2 (k+1)}a = b^{(k-1)j_2+j_1}a \in G_1.
\]
Now, it only remains to show that $ab=b^{k+1}a$. We observe that
\begin{equation*}
    \begin{split}
        b^{k+1}=& (c_0, c_{2k}, \ldots, c_{2m}, c_{2(2m+k)}, \ldots , c_{k-2},  c_{4k-4},c_k,c_{2k-2}, \ldots, c_{2n+k},c_{2(2n+k-1)}, \ldots, c_{2k-4},\\
        & c_{4k-10}, {\bf c_{4k-1}, c_{4k-6}}) (c_1, c_{2k-1}, \ldots, c_{2m+1}, c_{2(2m+k)-1}, \ldots , c_{k-1}, c_{4k-5},c_{k+1},c_{2k+1}, \ldots, \\
        & c_{2n+k+1},c_{2(2n+k)+1}, \ldots, c_{2k-3}, c_{4k-7}, {\bf c_{4k-2}, c_{4k-3}}).
    \end{split}
\end{equation*}
Let $c_t$ be an arbitrary element of $\F_q$. We shall divide our proof into two parts depending on the parity of $t$, namely, $t$ even and $t$ odd, respectively. \\

 \textbf{Case 1.} Let $t$ be an even positive integer. Then $t=2r$, where $0 \leq r \leq 2k-1$. Our aim is to show that $ab(c_{2r})= b^{k+1}a(c_{2r})$ for all $ r \in \{0, \ldots, 2k-1 \}$. It is easy to observe that when $\displaystyle   0 \leq r \leq \frac{k-2}{2}$ then 
\[
ab(c_{2r})= a(c_{2(2r+k-1)}) = c_{2(2r+k)-1}
\]
and
\[
b^{k+1}a(c_{2r}) = b^{k+1}(c_{2r+1}) = c_{2(2r+k)-1}.
\]
Similarly, when $\displaystyle  \frac{k}{2}\leq r \leq  k-2$ then we have
\[
ab(c_{2r})= a(c_{4r}) = c_{4r+1}
\]
and
\[
b^{k+1}a(c_{2r}) = b^{k+1}(c_{2r+1}) = c_{4r+1}.
\]
Thus, for all even $0 \leq t \leq  2(k-2)$, we have $ab(c_t)=b^{k+1}a(c_t)$. For $2(k-1) \leq t \leq  4k-2$, we divide the proof into the following three subcases:

\textbf{Subcase: 1.1.} Let $t=2(2s+k-1)$, where $0 \leq s \leq \frac{k-2}{2}$. In this case
\[
ab(c_{2(2s+k-1)})= a(c_{2s+2}) = c_{2s+3}
\]
and
\[
b^{k+1}a(c_{2(2s+k-1)}) = b^{k+1}(c_{2(2s+k)-1}) = c_{2s+3}.
\]

\textbf{Subcase: 1.2.} Suppose $t=2(2s+k)$, where $0 \leq s \leq \frac{k-6}{2}$. Then, we have
\[
ab(c_{2(2s+k)})= a(c_{2s+k+2}) = c_{2s+k+3}
\]
and
\[
b^{k+1}a(c_{2(2s+k)}) = b^{k+1}(c_{2(2s+k)+1}) = c_{2s+k+3}.
\]

\textbf{Subcase: 1.3.} We shall now verify the identity for the remaining elements of $\F_q$ with even indices, i.e., $t \in \{4k-8, 4k-4, 4k-2 \}$. It can be easily verified that
\begin{equation*}
    \begin{split}
        &ab(c_{4k-8}) = a(c_{4k-1}) = c_{4k-2} =b^{k+1}(c_{4k-7}) =b^{k+1}a(c_{4k-8}),\\
        &ab(c_{4k-4}) = a(c_0) = c_1 =b^{k+1}(c_{4k-3}) =b^{k+1}a(c_{4k-4}),\\
    \end{split}
\end{equation*}
and
\[
ab(c_{4k-2}) = a(c_{4k-5}) = c_{4k-6} =b^{k+1}(c_{4k-1}) =b^{k+1}a(c_{4k-2}).
\]
Thus, $ab(c_{2r})= b^{k+1}a(c_{2r})$ for all $ r \in \{0, \ldots, 2k-1 \}$.

\textbf{Case 2.} Let $t$ be an odd positive integer. Then $t=2r+1$, where $0 \leq r \leq 2k-1$. Here, we want to show that $ab(c_{2r+1})= b^{k+1}a(c_{2r+1})$ for all $ r \in \{0, \ldots, 2k-1 \}$. It is easy to observe that when $\displaystyle   0 \leq r \leq \frac{k-2}{2}$ then 
\[
ab(c_{2r+1})= a(c_{2(2r+k)+1)}) = c_{2(2r+k)}
\]
and
\[
b^{k+1}a(c_{2r+1}) = b^{k+1}(c_{2r}) = c_{2(2r+k)}.
\]
Similarly, if $\displaystyle  \frac{k}{2}\leq r \leq  k-2$, we have
\[
ab(c_{2r+1})= a(c_{4r-1}) = c_{4r-2}
\]
and
\[
b^{k+1}a(c_{2r+1}) = b^{k+1}(c_{2r}) = c_{4r-2}.
\]
Thus, for all odd $0 \leq t \leq  2k-3$, we have $ab(c_t)=b^{k+1}a(c_t)$. For odd $2k-1 \leq t \leq  4k-1$, we shall divide the proof into the following three subcases:

\textbf{Subcase: 2.1.} Let $t=2(2s+k)-1$, where $0 \leq s \leq \frac{k-6}{2}$. We have
\[
ab(c_{2(2s+k)-1})= a(c_{2s+k+3}) = c_{2s+k+2}
\]
and
\[
b^{k+1}a(c_{2(2s+k)-1}) = b^{k+1}(c_{2(2s+k)-2}) = c_{2s+k+2}.
\]

\textbf{Subcase: 2.2.} Let $t=2(2s+k)+1$, where $0 \leq s \leq \frac{k-2}{2}$. In this case
\[
ab(c_{2(2s+k)+1})= a(c_{2s+3}) = c_{2s+2}
\]
and
\[
b^{k+1}a(c_{2(2s+k)+1}) = b^{k+1}(c_{2(2s+k)}) = c_{2s+2}.
\]

\textbf{Subcase: 2.3.} We verify the identity below for the remaining elements of $\F_q$, i.e., $t \in \{4k-9, 4k-5, 4k-1 \}$.
\begin{equation*}
    \begin{split}
        &ab(c_{4k-9}) = a(c_{4k-2}) = c_{4k-1} =b^{k+1}(c_{4k-10}) =b^{k+1}a(c_{4k-9}),\\
        &ab(c_{4k-5}) = a(c_1) = c_0 =b^{k+1}(c_{4k-6}) =b^{k+1}a(c_{4k-5}),\\
    \end{split}
\end{equation*}
and
\[
ab(c_{4k-1}) = a(c_{4k-4}) = c_{4k-3} =b^{k+1}(c_{4k-2}) =b^{k+1}a(c_{4k-1}).
\]
Thus, $ab(c_t)= b^{k+1}a(c_t)$ for all $ t \in \{0, \ldots, 4k-1 \}$. Therefore, $G_1$ is a subgroup of the symmetric group ${\mathfrak S}_q$.

Next, we prove that the subgroup $G_1$ has exactly $q$ elements. This is equivalent to show that all the elements $b^ja^i$, where $j \in \{0, \ldots, 2k-1\}$ and $i \in \{0,1\}$, are distinct. Notice that $G_1= \{b^j, b^ja \mid  0 \leq j \leq 2k-1\}$. Since order of the element $b$ is $2k$ therefore all the elements $b^j$'s where $j \in \{0, \ldots, 2k-1\}$ are different. Similarly, all the elements of the form $b^ja$'s where $j \in \{0, \ldots, 2k-1\}$ are also different. Therefore, it only remains to show that $b^j \neq b^sa$ for any $j,s \in \{0, \ldots, 2k-1\}$. On the contrary, we assume that $b^{j} = b^{s}a$ for some $0 \leq j, s \leq 2k-1$ which implies  $a= b^{j-s}$. Since $ab =b^{k+1}a$, we have $b^{j-s}b=b^{k+1}b^{j-s} \implies b^{k}=I$ which leads to a contradiction as $|b|=2k$.  Consequently, $G_1$ has exactly $q$ elements, i.e, $|G_1|=q$.

We shall now show that every non-identity element of $G_1= \{b^j, b^ja \mid  0 \leq j \leq 2k-1\}$ does not have any fixed element. Since the permutation $b$ consists of two disjoint cycles each of length $2k$, the permutations $b^j$, $0 \leq j \leq 2k-1$, does not have any fixed element. Next, we consider the permutations of the form $b^ja$, where $ 0 \leq j \leq 2k-1$. The case $j=0$ is trivial. Now, let $1 \leq j \leq 2k-1$ be an integer and $c_t$ be an element of $\F_q$. We consider two cases, namely, $t$ is even and $t$ is odd.  We first assume that there exists $j \in \{0,1,\ldots,2k-1\}$ such that $b^{j}a(c_{t})=c_{t}$ for some $t$ even, then we have $b^j(c_{t+1}) = c_{t}$, which is not possible as there is no odd $r \in \{0,1,\ldots,q-1\}$ such that $b^j(c_{r}) = c_{r-1}$ for any $j$. Similarly, one can  prove that  $b^ja(c_t)\neq c_t$ when $t$ is odd, which completes the proof.
\end{proof}

The following theorem gives our second family of permutation group polynomials.
\begin{thm}\label{T32}
Let $q=2^m$, where $m\geq 3$ is a positive integer. Also, let
\[
a=(c_0,c_1)(c_2,c_3)\cdots(c_{q-2},c_{q-1}),
\]
and
\begin{equation*}
    \begin{split}
        b= &(c_{\frac{q}{2}-2},\ldots,c_{\frac{q}{2}-2(m+1)}, \ldots,c_{2},c_{0},  c_{\frac{q}{2}},\ldots, c_{\frac{q}{2}+2m},\ldots,c_{q-4},c_{q-2})\\
        &(c_{1},c_{3},\ldots, c_{2m+1}, \ldots,c_{\frac{q}{2}-1},c_{q-1},c_{q-3},\ldots, c_{q-(2m+1)}, \ldots,c_{\frac{q}{2}+1}),
    \end{split}
\end{equation*}
where $0 \leq m \leq \frac{q-4}{4}$, be permutations of $\F_q$. Then the set $G_2=\{b^{j}a^{i}: 0 \leq j \leq \frac{q}{2}-1,0\leq i \leq 1 \}$ is a subgroup of ${\mathfrak S}_q$ of order $q$ and none of its elements, except the identity, has a fixed point. 
\end{thm}
\begin{proof} Recall that, in order to show that $G_2$ is a subgroup of ${\mathfrak S}_q$, it is sufficient to show that $ab=b^{-1}a$.  Our aim is to show that $ab(c_t)=b^{-1}a(c_t)$ for all $t \in \{0,\ldots, q-1\}$.  Let $c_t \in \F_q$ be an element. We shall consider two cases, namely, $t$ even and $t$ odd.

\textbf{Case 1:} Let $0\leq t \leq q-2$ be an even integer, i.e., $t=2i$ where $0 \leq i \leq \frac{q-2}{2}$. Notice that when $1\leq i \leq \frac{q-4}{4}$ then 
\[
ab(c_{2i})= a(c_{2(i-1)}) =c_{2i-1}=b^{-1}(c_{2i+1})=b^{-1}a(c_{2i}).
\]
Similarly, when $\frac{q}{4} \leq i \leq \frac{q-4}{2}$ then 
\[
ab(c_{2i})= a(c_{2(i+1)}) =c_{2i+3}=b^{-1}(c_{2i+1})=b^{-1}a(c_{2i}).
\]
One can easily verify that
\[
ab(c_0)= a(c_{\frac{q}{2}}) = c_{\frac{q+2}{2}} = b^{-1}(c_1)=b^{-1}a(c_0) 
\]
and
\[
ab(c_{q-2})= a(c_{\frac{q-4}{2}}) = c_{\frac{q-2}{2}} = b^{-1}(c_{q-1})=b^{-1}a(c_{q-2}).
\]
Thus, $ab(c_t)=b^{-1}a(c_t)$ for all even $t \in \{0,\ldots, q-1\}$.

\textbf{Case 2:} Let $1\leq t \leq q-1$ be an odd integer. The proof follows along the similar lines as in the Case~1.

The above two cases imply $ab=b^{-1}a$. It is easy to see that $|a|=2, |b|=\frac{q}{2}$, so it is a Dihedral group of order $2\times \frac{q}{2}=q$. Thus, it only remains to show that  every non-identity element of $G_2$ has no fixed points. The permutation of the form $b^j$, where $ j\in \{0,1,\ldots,\frac{q}{2}-1\}$, does not fix any element of $\F_q$ by the similar argument used in Theorem \ref{T31}. Now consider the permutation of the form $b^ja$, where $ j\in \{0,1,\ldots,\frac{q}{2}-1\}$. In this case, we have $$b^ja(c_t)=b^j(c_{t\pm1}) \neq c_t$$ 
as the map $b^j$ does not send the element $c_{t\pm1}$ to $c_t$ for any $t \in \{0, 1, \ldots, q-1\}$, which completes the proof.
\end{proof}
The next theorem introduces the third family of permutation group polynomials.
\begin{thm}\label{T33}
Let $q=4k$ be a positive integer, where $k=2^{\ell}$ for some positive integer $\ell \geq 2$. Also, let 
\[
a=(c_0,c_1)(c_2,c_3)\cdots(c_{q-2},c_{q-1})
\]
and
\begin{equation*}
    \begin{split}
        b=& (c_0, c_{2k+1}, \ldots, c_{2i}, c_{2i+2k+1}, \ldots , c_{k-2}, c_{3k-1},{\bf c_{k+1},c_{3k}},c_{k+2},c_{3k+2}, \ldots, c_{2(j+1)+k},c_{2(j+1)+3k},\\
        & \ldots, c_{2k-2}, c_{4k-2}) (c_{3k-2},c_{2k-1}, \ldots, c_{3k-2-2j}, c_{2k-1-2j}, \ldots ,c_{2k+2},c_{k+3},{\bf c_{2k},c_{k},} c_{4k-1}, c_{k-1}, \\
        & \ldots, c_{4k-1-2i},c_{k-1-2i}, \ldots, c_{3k+1},c_{1}), 
    \end{split}
\end{equation*}
where $i \in \{0,1,\ldots,\frac{k-2}{2}\}$ and $j \in \{0,1,\ldots,\frac{k-4}{2}\}$, be two permutations of ${\mathfrak S}_q$.
Then the set $G_3=\left <a,b : |a|=2, |b|=2k \text{ and } ab=b^{k-1}a\right >$ is a subgroup of ${\mathfrak S}_q$ of order $q$ in which no permutation has any fixed point except the identity.
\end{thm}
\begin{proof}
First, we shall show that $ab=b^{k-1}a$ which will play prominent role to prove that all the elements $b^{j_1}$ and $b^{j_2}a$ are different, where $0 \leq j_1, j_2 \leq 2k-1$, i.e.,  $|G_3|=q$. We observe that 
\begin{equation*}
\begin{split}
b^{k-1}=&(c_{3k-1},c_{2k-2},\ldots,c_{3k-(2j+1)},c_{2k-2(j+1)},\ldots,c_{2k+3},c_{k+2},{\bf c_{2k+1},c_{k+1}},c_{4k-2},
c_{k-2},\ldots,  c_{4k-2(i+1)},\\&c_{k-2(i+1)},\ldots,c_{3k},c_{0}) ({\bf c_{k},c_{3k+1}}, c_{k+3}, c_{3k+3},\ldots, c_{k+2(j+1)+1},c_{2k+2(j+1)+1},c_{2k-3},c_{4k-3},c_{2k-1},\\&c_{4k-1},c_{1},c_{2k},\ldots,c_{2i+1},c_{2(i+k)}, \ldots,c_{k-1},c_{3k-2}).
\end{split}
\end{equation*}
Let $c_t$ be an arbitrary element of $\F_q$. We shall split our proof into two parts depending on the parity of $t$, namely, $t$ even and $t$ odd, respectively.

\textbf{Case 1:} Let $ t=2r$, where $r \in \{0,1,\ldots,2k-1\}$. Further, we divide this case into the following subcases:

\textbf{Subcase 1.1:} When $r \in \{0,1,\ldots, \frac{k-2}{2}\}$ then $b$ sends $c_{2r}$  to $c_{2r+2k+1}$, therefore, 
\[ab(c_{2r})=a(c_{2r+2k+1})=c_{2r+2k}.\] 
Now, 
\[b^{k-1}a(c_{2r})=b^{k-1}(c_{2r+1})=c_{2r+2k}.\]

\textbf{Subcase 1.2:} Let $r \in \{\frac{k}{2}+1, \frac{k}{2}+1,\ldots,k-1\}$.  Then  using the definition of $b$ and $b^{k-1}$, we have
 \[ab(c_{2r})=a(c_{2r+2k})=c_{2r+2k+1}=b^{k-1}(c_{2r+1})=b^{k-1}a(c_{2r}).\] 
 
\textbf{Subcase 1.3:} If $r \in \{k+1,k+2,\ldots,\frac{3}{2}k-1\}$ then  
\[ab(c_{2r})=a(c_{2r-k+1})=c_{2r-k}. \] Moreover, we have 
\[b^{k-1}a(c_{2r})=b^{k-1}(c_{2r+1})=c_{2r+1-k-1}=c_{2r-k}.\]

\textbf{Subcase 1.4:} Let  $r \in \{\frac{3k}{2},\frac{3k}{2}+1,\ldots, 2k-2\}$. In this subcase, we have
\[ab(c_{2r})=a(c_{2r-2k+2})=c_{2r-2k+3}=b^{k-1}(c_{2r+1})=b^{k-1}a(c_{2r}).\]

\textbf{Subcase 1.5:} We can easily verify that
\[ 
ab(c_{k})=a(c_{4k-1})=c_{4k-2}=b^{k-1}(c_{k+1})=b^{k-1}a(c_{k}),
\]
\[
ab(c_{2k})=a(c_{k})=c_{k+1}=b^{k-1}(c_{2k+1})=b^{k-1}a(c_{2k}),
\]
and
\[
ab(c_{4k-2})=a(c_{0})=c_{1}=b^{k-1}(c_{4k-1})=b^{k-1}a(c_{4k-2}).
\]

\textbf{Case 2:} In this case, we assume that  $t=2r+1$ and $r \in \{0,1,\ldots,2k-1\}$. We shall again divide this case into the following subcases.

\textbf{Subcase 2.1:} Let $r \in \{1,\ldots, \frac{k}{2}-1\}$. Then we have, 
\[ab(c_{2r+1})=a(c_{2r+1+3k-2})=c_{2r+3k-2}.\] 
Now, 
\[b^{k-1}a(c_{2r+1})=b^{k-1}(c_{2r})=c_{2r+3k-2}.\]

\textbf{Subcase 2.2:} If $r \in \{\frac{k}{2}+1, \frac{k}{2}+2, \ldots, k-1\}$ then $b$ maps $c_{2r+1}$ to $c_{2r+1+k-3}=c_{2r+k-2}$, therefore, 
\[ab(c_{2r+1})=a(c_{2r+k-2})=c_{2r+k-1},\] 
and 
\[b^{k-1}a(c_{2r+1})=b^{k-1}(c_{2r})=c_{2r+k-1}.\]

\textbf{Subcase 2.3:} When $r \in \{k,k+1,\ldots,\frac{3}{2}k-2\}$ then  
\[ab(c_{2r+1})=a(c_{2r+1-(2k-1)})=c_{2r-2k+3}=b^{k-1}(c_{2r})=b^{k-1}a(c_{2r+1}). \] 

\textbf{Subcase 2.4:} Let  $r \in \{\frac{3}{2}k,\frac{3}{2}k+1,\ldots, 2k-1\}$. In this subcase, we have 
\[ab(c_{2r+1})=a(c_{2r+1-3k})=c_{2r-3k}=b^{k-1}(c_{2r})=b^{k-1}a(c_{2r+1}).\] 

\textbf{Subcase 2.5:} One can easily verify for the remaining elements of $\F_q$, namely, $c_{1}$, $c_{k+1}$ and $c_{3k-1}$ that 
\[ 
ab(c_{1})=a(c_{3k-2})=c_{3k-1}=b^{k-1}(c_{0})=b^{k}a(c_{1}),
\]
\[ 
ab(c_{k+1})=a(c_{3k})=c_{3k+1}=b^{k-1}(c_{k})=b^{k}a(c_{k+1}),\text{ and}
\]
\[
ab(c_{3k-1})=a(c_{k+1})=c_{k}=b^{k-1}(c_{3k-2})=b^{k-1}a(c_{3k-1}).
\]

From the above two cases,  we have $ab=b^{k-1}a$ which implies that every element of $G_3$ can be written as $b^{j}a^{i}$, where $j \in \{0,1,\ldots,2k-1\}$ and $i \in \{0,1\}$. Next, we shall show that $|G_3|=q$. First, it is clear that $b^{j}$, $j \in \{0,1,\ldots,2k-1\}$, are all distinct elements of $G_3$ as $b$ is a product of two disjoint cycles each of length $2k$. Similarly, all $b^{j}a$, $j \in \{0,1,\ldots,2k-1\}$, are different. For the remaining case, on the contrary, we assume that $b^{j_1}=b^{j_2}a$ for some $0 \leq j_1,j_2 \leq 2k-1$, which implies that $ a=b^{j_1-j_2}$. Now  using the above derived relation $ab=b^{k-1}a$, we have $b^{j_1-j_2}b=b^{k-1}b^{j_1-j_2}$, this gives  $I=b^{k-2},$ which is a contraction. Therefore, $b^{j_1}a^{i_1} \neq b^{j_2}a^{i_2}$ for all $j_1,j_2 \in \{0,1, \ldots, 2k-1\}$ and $i_1,i_2 \in \{0,1\}$, which gives $|G_3|=q$. 

Thus, it only remains to show that non-identity elements of $G_3$ have no fixed points.  By the similar argument used in Theorem \ref{T31}, $b^{j}$, where $j \in \{0,1,\ldots,2k-1\}$,  does not fix any element of $\F_q$. Now we shall show that $b^ja$, $j \in \{0,1,\ldots,2k-1\}$, will not fix any element of $\F_q$. Let $c_t$ be an element of $\F_q$ and $b^ja \in G_3$ be a permutation, where $j \in \{0,1,\ldots,2k-1\}$. First, let $0 \leq t \leq 4k-2$ be an even integer then $b^ja(c_{t})=b^j(c_{t+1})$. Therefore, for some $t$ even, $c_t$ is fixed by $b^ja$ if and only if $b^j(c_{t+1})=c_t$. Now, from the cyclic structure of $b$, either there does not exist such a $j$ or if it exists then $t+1 \in \{k,k+2,2k,3k\}$. This leads to a contradiction as $t$ is even. Similarly, by using the definition of $b$, we can handle the case when $t$ is odd.
\end{proof}

\begin{rmk}\label{Not_Abelian}
The subgroups $G_1$, $G_2$ and $G_3$ in Theorem \ref{T31}, Theorem \ref{T32} and Theorem \ref{T33}, respectively, are not Abelian unlike the subgroup $G$ in Lemma \ref{JJ_tpl}.
\end{rmk}

\begin{thm}
The permutation group polynomials (say $f_1,f_2$ and $f_3 \in \F_q[X,Y]$) corresponding to $G_1$, $G_2$ and $G_3$, respectively, are not equivalent to the only known family of permutation group polynomials and are also not equivalent among themselves.
\end{thm}

\begin{proof}  First, we shall show that $f_1,f_2$ and $f_3$ are inequivalent to the only  known family of permutation group polynomials, called the $e$-Klenian polynomials, given by Gutierrez and Urroz  (see; Lemma~\ref{JJ_tpl}). Let $f\in \F_q[X,Y]$ be an $e$-Klenian polynomial and $G$ be the subgroup corresponding to the permutation polynomial tuple of $f$. Now, on the contrary, we assume that $f$ is equivalent to $f_1$ which implies that there exist $\sigma, \lambda \in {\mathfrak S}_q$ such that  $\sigma G  \lambda= G_1$. Therefore, $\sigma a_1 \lambda=I$ for some $a_1 \in G$ which implies $G_1=\sigma G\sigma^{-1}$. Now we shall prove $\sigma G \sigma^{-1}$ is an Abelian subgroup of ${\mathfrak S}_q$. Let $\sigma a_1 \sigma^{-1}$ and $\sigma a_2 \sigma^{-1}$ be two elements of $\sigma G \sigma^{-1}$. Then
\[
(\sigma a_1\sigma^{-1}) (\sigma a_2\sigma^{-1})= \sigma a_1 a_2 \sigma^{-1}=\sigma a_2 a_1\sigma^{-1}=(\sigma a_2 \sigma^{-1}) (\sigma a_1 \sigma^{-1}).
\]
This implies $\sigma G \sigma^{-1}=G_1$ is an Abelian subgroup, which is a contradiction. Therefore, $f_1$ is inequivalent to $f$. Using the similar arguments, we can say that $f_2$ and $f_3$ are inequivalent to $f$.

 Now we shall prove that $f_1,f_2$ and $f_3$ are inequivalent among themselves. From the above discussion, it is sufficient to show that their corresponding subgroups are not conjugate to each other. To this end, we will prove that the number of  elements having order $2$ in each subgroup are different. We have $$G_1=\{b^{j}a^{i}: 0 \leq j \leq 2k-1=\frac{q}{2}-1, 0\leq i \leq 1\},$$ where $ab=b^{k+1}a, |a|=2, |b|=2k$, $k=2^{\ell}$ for some positive integer $\ell \geq 2$ and $q=4k$. Now if $i=0$ then there is only one element of the  form $b^{j}$ of order $2$ which is $b^{k}$. Next, we suppose that $|b^ja|=2$ for some $j \in \{0,1,\ldots,2k-1=2^{\ell+1}-1\}$ then  $b^jab^ja=I$. Now using the given relation $ab=b^{k+1}a$, we have $b^{(k+2)j}=I,$ which is true if and only if   $2k \mid (k+2)j$, i.e., $2^{\ell+1}\mid (2^{\ell}+2)j=(2^{\ell-1}+1)2j,$ which is again equivalent to $2^{\ell+1}\mid 2j$ as $\ell\geq 2$ and $\gcd(2^{\ell-1}+1,2^{\ell+1})=1$.  As $2^{\ell+1}\mid 2j$ if and only if $j=0$ or $j=2^{\ell}$ as $0 \leq j<2^{\ell+1}$. So, there are exactly three elements of order two in $G_1$. Now, since $G_2$ is a Dihedral group, so, it has $\frac{q}{2}+1$ elements of order two. Next, we count the number of elements of order two in $$G_3=\{b^{j}a^{i}: 0 \leq j \leq 2k-1=\frac{q}{2}-1, 0\leq i \leq 1, ab=b^{k-1}a, |a|=2, |b|=2k\},$$ where $k=2^{\ell}$ for some positive integer $\ell \geq 2$. Here, again we have only one element of order two of the form $b^j$ by the same argument. Now, let $b^ja$ be an element of order two in $G_3$, where $j \in \{ 0,1, \ldots,2k-1=2^{\ell+1}-1\}$. Then we have $b^jab^ja=I$, which is equivalent to $b^{kj}=I$ as $ab=b^{k-1}a$ and $|a|=2$. The equality $b^{kj}=I$  is true if and only if  $2k=2^{\ell+1}\mid kj=2^{\ell}j$, which is possible if and only if $j$ is even. Therefore, there are exactly $k+1=\frac{q}{4}+1$ elements of order two in $G_3$.  Thus, for $q\geq 16$ the subgroups $G_1$, $G_2$ and $G_3$ have different number of elements of order two, which concludes that $G_1$, $G_2$ and $G_3$ are not conjugate to each other and consequently  $f_1$, $f_2$ and $f_3$ are inequivalent to each other.
\end{proof}

\section{Companions of local permutation polynomials}\label{S4}
In this section, we investigate the existence of companions for bivariate LPPs. As noted in~\cite[Theorem 36]{JJ_2023}, a companion provides an effective algebraic approach for constructing a complete set of MOLS, which have a wide range of applications (see, for example, \cite{KD_2015,MGFL_2020,SS_1992,WW_2014}). Gutierrez and Urroz~\cite{JJ_2023} introduced a family of permutation group polynomials known as $e$- Klenian polynomials over $\F_q$ and they proved the existence of companions for these polynomials when the characteristic of $\F_q$ is odd. However, their results for finite fields with even characteristic were primarily experimental. We show that there are no companions for $0$- Klenian polynomials and give explicit expressions for companions of $e$- Klenian polynomials over finite fields with even characteristic for all $e \geq 1$. Moreover, we present companions for each of the families of permutation group polynomials discussed in Section~\ref{S3}. The following definitions will be frequently used in our results.
\begin{defn}\label{D45}
 Let $h_1,h_2\in\F_q[X]$ be two permutation polynomials. Let $A :=\{(c,h_1(c)):c \in \F_q\}$ and $B:=\{ (c,h_2(c)):c \in \F_q\}$. We say that $h_1$ intersects $h_2$ simply  if $|A \cap B|=1$. 
 \end{defn}
 
 \begin{defn}
Let $f\in\F_q[X,Y]$ be an LPP and $\underline{\beta}_f=(\beta_0,\ldots,\beta_{q-1}) \in {\mathfrak S}_q^q$ be the corresponding permutation polynomial tuple. We say that a univariate permutation polynomial $h$ intersects the bivariate LPP $f$ simply if $h$ intersects $\beta_i$ simply for each $0 \leq i \leq q-1$. 
 \end{defn}

The following lemma translates the problem of the investigation of the orthogonality of two LPPs to their corresponding permutation polynomials tuples.
\begin{lem}\label{ortho_tpl}
Let $\underline{\beta}_f =(\beta_0,\ldots,\beta_{q-1})$ and $\underline{\gamma}_g= (\gamma_0,\ldots,\gamma_{q-1})$ be the permutation polynomial tuples corresponding to the LPPs $f$ and $g \in \F_q[X,Y]$, respectively. Then $f$ and $g$ are orthogonal if and only if $\beta_i$ intersects $\gamma_k$ simply for all $i,k \in \{0, \ldots, q-1\}$. 
\end{lem}
\begin{proof}
Our aim is to show that $f$ and $g$ are orthogonal if and only if $\beta_i$ intersects $\gamma_k$ simply for all $i,k \in \{0, \ldots, q-1\}$. Consider the following sets 
  \begin{equation*}\label{E42}
 \begin{split}
 A_i=\{(c,\beta_i(c)): c\in \F_q\}, \text{ and }
 B_k=\{(c,\gamma_k(c)):c \in \F_q\}, \text{ for } 0 \leq i,k \leq q-1.
 \end{split}
 \end{equation*}
 Since $\underline{\beta}_f$ and $\underline{\gamma}_g$ are permutation polynomials tuples of $f$ and $g$, respectively. Therefore,  $f$ maps each element of $A_i$ to $c_i \in  \F_q$ and $g$ maps each element of $B_k$ to $c_k \in \F_q$. Now, $f$ and $g$ are orthogonal if and only if  $|A_i \cap B_k|=1$ for all $i,k \in \{0, \ldots, q-1\}$. Equivalently, $\beta_i$ and $\gamma_k$ intersect simply for all $i,k \in \{0, \ldots, q-1\}$.
\end{proof}
Next, we present a crucial lemma for determining whether a permutation group polynomial 
$f \in \F_q[X,Y]$ has a companion. This lemma simplifies the problem by reducing it to find a univariate polynomial that intersects $f$ simply.
\begin{lem}\label{T43}
Let $f\in \F_q[X,Y]$ be a permutation group polynomial and $h\in\F_q[X]$ be a permutation polynomial which intersects $f$ simply. Let $\underline{\beta}_f=(\beta_0,\ldots,\beta_{q-1})$ be the permutation polynomial tuple corresponding  to $f$. Then the polynomial $g$ associated  to $\underline{\gamma}_g=(h \beta_0, \ldots,h \beta_{q-1})$ is a companion of $f$.
\end{lem}
\begin{proof}
Notice that $g$ is equivalent to $f$ and hence it is also an LPP. Consider the set $A=\{(c,h(c)):c \in \F_q\}$. Since $h$ intersects $f$ simply, we have $|A \cap \{(c,\beta_i(c)):c\in \F_q\}|=1$ for all $i \in \{0, \ldots, q-1\}$. This is equivalent to say that $h(X)=\beta_i(X)$ has a unique solutions for all $i \in \{0, \ldots, q-1\}$, i.e., $h\beta_i^{-1}$ has a unique fixed point for all $i \in \{0, \ldots, q-1\}$. Now, from Lemma \ref{ortho_tpl}, it is sufficient to show that each permutation $h\beta_i$ of $\underline{\gamma}_g$ intersects each member $\beta_j$ of $\underline{\beta}_f$ simply, where $i,j \in \{0, \ldots, q-1 \}$. That is $(h \beta_i) \beta_j^{-1}=h(\beta_j \beta_i^{-1})^{-1}$ has a unique fixed point for all $i,j \in \{0, \ldots, q-1 \}$. Since $f$ is a permutation group polynomial, therefore, $ \beta_j \beta_i^{-1}=\beta_t$ for some $0 \leq t \leq q-1$ and $(h \beta_i) \beta_j^{-1}=h\beta_t^{-1}$, which has a unique fixed point from the above discussion.
\end{proof}
 Investigating whether a given permutation group polynomial has a companion is generally a difficult problem. In~\cite{JJ_2023}, the authors determined the companions for the $e$- Klenian polynomials over finite fields of odd characteristic. However, the problem of finding companions for $e$- Klenian polynomials over finite fields of even characteristic was left open. We resolve this problem completely in the next two theorems. In the following theorem, we show that $0$- Klenian polynomials do not have companions over finite fields of characteristic two.

\begin{thm}\label{T44}
Let $q=2^m$, where $m$ is a positive integer. Then $0$- Klenian polynomial over $\F_q$ has no companion.
\end{thm} 
\begin{proof}
Let $f$ be a $0$- Klenian polynomial given in Lemma~\ref{JJ_tpl}. We may assume, without loss of generality, that the corresponding permutation polynomial tuple can be written as {\small $\underline{\beta}_f=(I,b,b^2,\ldots,b^{q-1})$,} where $b=(c_0,c_1,\ldots,c_{q-1})\in {\mathfrak S}_q$. Our aim is to show that $f$ does not have a companion. On the contrary, we assume that $g$ is a companion of $f$ and $\underline{\gamma}_g=(\gamma_0,\ldots,\gamma_{q-1})$ be the corresponding permutation polynomial tuple. Then from Lemma~\ref{ortho_tpl}, $\gamma_k$ intersects $b^i$ simply for all $i,k \in \{0, \ldots, q-1\}$. In the particular case when $k=0$, i.e., $\gamma_0$ intersects $b^i$ simply for all $i\in \{0, \ldots, q-1\}$, which implies that 
\[
\gamma_0=(c_0,c_{0+i_0})(c_1,c_{1+i_1}) \cdots (c_{q-1},c_{q-1+i_{q-1}}),
\]
where $\{i_0,i_1,\ldots,i_{q-1}\}=\{0,1,\ldots,q-1\}$ and indices are taken modulo $q$. Since $\gamma_0$ is a permutation, $\{0+i_0,1+i_1,\ldots, q-1+i_{q-1}\}$ is a transversal in the Cayley table of $\mathbb{Z}_q$, which is a contradiction. 
\end{proof}
In the following theorem, we prove the existence of a companion for  $e$- Klenian polynomials over finite fields of even characteristic, where $e \geq 1$.  
\begin{thm}\label{T45}
Let $q=2^m$, $1 \leq e < m$, $\ell=2^e, t=\dfrac{q}{\ell}=2^{m-e}$, where $m$ and $e$ are positive integers. Moreover, let  $f\in \F_q[X,Y]$ be an $e$-Klenian polynomial over $\F_q$. Then $f$ has a companion.
\end{thm}

\begin{proof}
Without loss of generality, we can assume that $ \ell \mid t$. Now, let $\underline{\beta}_f=(\beta_0,\ldots,\beta_{q-1})$ be the permutation polynomial tuple of the $e$-Klenian polynomial $f$. Then from Lemma~\ref{JJ_tpl}, we have
\[
\{\beta_0,\ldots,\beta_{q-1}\}=\{\beta^{v}\alpha^{u} \mid 0 \leq u \leq \ell-1, 0 \leq v \leq t-1\},
\]
where $\alpha$ and $\beta$ are as defined  in Lemma \ref{JJ_tpl}. Our aim is to construct a permutation polynomial $h\in \F_q[X]$ such that $h$ intersects $\beta_k$ simply, for all $k \in \{0, \ldots, q-1\}$. Then using Lemma \ref{T43}, we can show that the bivariate polynomial $g$ corresponding to $(h\beta_0,\ldots,h\beta_{q-1})$ is a companion of $f$. For any $ i \in \{0,1,\ldots,\ell-1\}$ and $j \in \{0, \ldots, t-1\}$, we define $h$ as follows:
\[
h(c_{i+j\ell})=
\begin{cases}

       c_{\pi_0(i)+\ell  ((j(\ell+1)+\pi_0(i))\hspace{-.2cm}\mod t)}& \text{if }0 \leq j \leq \dfrac{t}{\ell}-1, \\
       c_{\pi_1(i)+\ell((j(\ell+1)+\pi_1(i))\hspace{-.2cm}\mod t)}& \text{if } \dfrac{t}{\ell}\leq j \leq 2\dfrac{t}{\ell}-1, \\
       \vdots & \vdots\\
       c_{\pi_{\ell-1}(i)+\ell((j(\ell+1)+\pi_{\ell-1}(i))\hspace{-.2cm}\mod t)}&\text{if }  (\ell-1) \dfrac{t}{\ell} \leq j \leq \ell \dfrac{t}{\ell}-1=t-1,   \\
       	\end{cases}
\]
where $\pi_s:\mathbb{Z}_{\ell} \rightarrow \mathbb{Z}_{\ell}$ such that $\pi_s(i)=(i+s) \pmod \ell$, $s \in \{0,1,\ldots,\ell-1\}$ and $\mathbb{Z}_{\ell}$ is the group of integer modulo $\ell$ under addition. 
First, we shall prove that $h$ is a permutation of $\F_q$. On the contrary, we assume that $$h(c_{i_1+j_1\ell})=h(c_{i_2+j_2\ell})$$ for some $(i_1,j_1) \neq (i_2,j_2)$, where $ i_1,i_2 \in \{0,1,\ldots,\ell-1\}$ and $j_1,j_2 \in \{0, \ldots, t-1\}$ which implies $$c_{\pi_{s_1}(i_1)+\ell((j_1(\ell+1)+\pi_{s_1}(i_1))\hspace{-.2cm}\mod t)}=c_{\pi_{s_2}(i_2)+\ell((j_2(\ell+1)+\pi_{s_2}(i_2))\hspace{-.2cm}\mod t)}.$$ However,  it is possible only if the following two conditions hold true
  \[\pi_{s_1}(i_1)=\pi_{s_2}(i_2)\]
   and
    \[j_1(\ell+1)+\pi_{s_1}(i_1) \equiv j_2(\ell+1)+\pi_{s_2}(i_2) \pmod t.\]
    Using these two conditions, we have  
    \[j_1(\ell+1) \equiv j_2(\ell+1) \pmod t\]
     this gives $j_1 \equiv j_2 \pmod t$, which is equivalent to $j_1=j_2$ as $j_1,j_2 \in \{0, \ldots, t-1\}$. The equality of $j_1$ and $j_2$ implies $s_1=s_2$. Now since $s_1=s_2$, we have $i_1 \equiv i_2 \pmod \ell$ which implies $i_1=i_2$. Therefore, we have $(i_1,j_1) = (i_2,j_2)$, which is a contradiction.

We shall now prove that $h(c)=\beta_k(c)$ for exactly one $c\in \F_q$ for all $k \in \{0, \ldots, q-1\} $. Since both $h$ and $\beta_k$'s, where $k \in \{0,1, \ldots, q-1 \}$, are permutations, it is equivalent to show that for each fixed $c \in \F_q$, there exists a unique $k \in \{0,1, \ldots, q-1\}$ such that $h(c)=\beta_k(c)$ and for each fixed $k \in \{0,1, \ldots, q-1\}$, there exists a unique $c \in \F_q$ such that $h(c)=\beta_k(c)$. The first assertion directly follows from the property of the permutation polynomial tuple as for any $k_1, k_2 \in \{0, \ldots, q-1\}$, $k_1 \neq k_2$ the condition $\beta_{k_1}(c)= \beta_{k_2}(c)$ would imply that $c$ is a fixed point of the permutation $\beta_{k_2}^{-1}\beta_{k_1}$, which is not possible as $(\beta_0,\ldots,\beta_{q-1})$ is a permutation polynomial tuple. We shall now show that for each fixed $k \in \{0,1, \ldots, q-1\}$, there exists a unique $c \in \F_q$ such that $h(c)=\beta_k(c)$. Equivalently, for any fixed $u \in \{0,1,\ldots, \ell-1\}$ and $v\in \{0,1,\ldots,t-1\}$, there exists a unique $i \in \{0,1,\ldots, \ell-1\}$ and a unique $j\in \{0,1,\ldots,t-1\}$ such that $a^ub^v(c_{i+\ell j}) =h(c_{i+\ell j})$.
Notice that for any $u \in \{0,1,\ldots, \ell-1\}$ and $v \in \{0, \ldots, t-1\}$, we have
\begin{equation}\label{E44}
   a^u b^v(c_{i+\ell j})=c_{(u+i) \hspace{-.2cm}\pmod \ell +\ell ((v+j)\hspace{-.2cm}\pmod t)} =c_{\pi_u(i) +\ell ((v+j)\hspace{-.2cm}\pmod t)}.   
\end{equation}
We also observe that 
\begin{equation}\label{E43}
h(c_{i+\ell j})=c_{\pi_{s}(i)+\ell((j(\ell+1)+\pi_{s}(i))\hspace{-.2cm}\mod t)}
\end{equation}
 for some $s \in \{0,1,\ldots,\ell-1\}$ if and only if $j \in   \left \{ \dfrac{st}{\ell}, \dfrac{st}{\ell}+1,\ldots, \dfrac{(s+1)t}{\ell}-1\right\}$. Now, we assume that  $a^u b^v(c_{i+\ell j})=h(c_{i+\ell j})$. Then from Equation \eqref{E44} and Equation \eqref{E43}, we have  that $$s=u  (\pi_{s}=\pi_u), j \in   \left \{ \dfrac{ut}{\ell}, \dfrac{ut}{\ell}+1,\ldots, \dfrac{(u+1)t}{\ell}-1\right \},$$ and 
\begin{equation}\label{E42}
 (v+j) \equiv j(\ell+1)+\pi_{u}(i) \pmod t.
\end{equation}
  The Equation \eqref{E42} is equivalent to 
\begin{equation}\label{E41}
 v \equiv j \ell+\pi_{u}(i) \pmod t.
 \end{equation}
Next, we suppose that $v \equiv w \pmod \ell$ for some $w \in \{0,1,\ldots,\ell-1\}$. Then we have $ \pi_{u}(i)=w$ as $v \equiv   \pi_{u}(i) \pmod \ell$, where the latter condition follows from the fact that $\ell \mid t$. Since $i \in \{0,1,\ldots,\ell-1\}$, so the equivalence $i \equiv w-u \pmod \ell $ has a unique solution for $i$. Next, our aim is to determine  $j$ uniquely. From Equation \eqref{E41}, we have  $v-\pi_{u}(i) \equiv j \ell \pmod t$ which implies that  $\frac{v-\pi_{u}(i)}{\ell} \equiv j  \pmod {\dfrac{t}{\ell}}$. Therefore,  the equivalence $\frac{v-\pi_{u}(i)}{\ell} \equiv j  \pmod {\dfrac{t}{\ell}}$ has a unique solution for $j$ as  $j \in   \left \{ \dfrac{ut}{\ell}, \dfrac{ut}{\ell}+1,\ldots, \dfrac{(u+1)t}{\ell}-1\right \}$. Thus, we obtain a unique ordered pair $(i,j)$ for fixed $(u,v)$, which implies that for each fixed $k\in \{0,1, \ldots, q-1\}$, there exists a unique $c \in \F_q$ such that $h(c)=\beta_k(c)$. 
\end{proof}
Next, we prove the following lemma which plays prominent role in the subsequent results.
\begin{lem}\label{L45}
Let $2\leq t \leq q$ be an even integer and $C=(c_{i_0},c_{i_1},\ldots,c_{i_{t-1}})$ be a $t$-cycle. Also, let 
\begin{equation*}
\begin{split}
A&=\left \{C^{t-2}(c_{i_0}),C^{t-4}(c_{i_1}),\ldots,C^{0}(c_{i_{\frac{t-2}{2}}}),C^{t-2}(c_{i_{\frac{t}{2}}}),C^{t-4}(c_{i_{\frac{t+2}{2}}}),\ldots,C^{0}(c_{i_{t-1}})\right \} \text{ and}\\
B&=\left \{C^{0}(c_{i_0}),C^{2}(c_{i_1}),\ldots,C^{t-2}(c_{i_{\frac{t-2}{2}}}),C^{0}(c_{i_{\frac{t}{2}}}),C^{2}(c_{i_{\frac{t+2}{2}}}),\ldots,C^{t-2}(c_{i_{t-1}})\right \}.
\end{split}
\end{equation*}
Then $|A|=|B|=t$. 
\end{lem}
\begin{proof}
First, we shall prove that $|A|=t$. We observe that $A=A_1\cup A_2$, where $A_1=\{C^{t-2-2s}(c_{i_s}): 0 \leq s \leq \frac{t-2}{2}\}$ and $A_2=\{C^{t-2-2(s-\frac{t}{2})}(c_{i_s}): \frac{t}{2} \leq s \leq t-1\}$. We now show that all the elements of $A_1$ are distinct. Let $C^{t-2-2s_1}\left(c_{i_{s_1}}\right)=C^{t-2-2s_2}\left(c_{i_{s_2}}\right)$ for some $ s_1,s_2 \in \left\{ 0,\ldots, \frac{t-2}{2}\right\}$. Then  $c_{(t-2-s_1)\mod t}=c_{(t-2-s_2)\mod t}$ which is true if and only if $t-2-s_1 \equiv t-2-s_2 \mod t$. Therefore, $s_1 \equiv s_2 \mod t$ implies that $s_1=s_2$ which proves the assertion. Similarly, we can prove that all the elements of $A_2$ are distinct. Now, let $C^{t-2-2s_1}(c_{i_{s_1}})\in A_1$ and $C^{t-2-2(s_2-\frac{t}{2})}(c_{i_{s_2}})\in A_2$ such that $C^{t-2-2s_1}(c_{i_{s_1}})=C^{t-2-2(s_2-\frac{t}{2})}(c_{i_{s_2}})$ for some $ s_1 \in \left\{0,\ldots, \frac{t-2}{2}\right\} $ and $s_2 \in \left\{\frac{t}{2},\ldots, t-1 \right\}$. Then we have $c_{(t-2-s_1)\mod t}=c_{(t-2-s_2-t)\mod t}$, therefore,  $t-2-s_1 \equiv t-2-s_2-t \mod t$ implies that $s_1=s_2$, which is a contradiction. Therefore all the elements of $A$ are distinct. Other one can be proved using the same argument.
\end{proof}
Now, we construct a companion for permutation group polynomial in Theorem \ref{T31}.
 \begin{thm}\label{T46}
 Let $q=4k$, where $k=2^{\ell}$ for some integer $\ell  \geq 2$ and  $f\in \F_q[X,Y]$ be a permutation group polynomial given in Theorem \ref{T31}. Then $f$ has a companion.
 \end{thm}
 \begin{proof}
  Recall that the permutations in the permutation polynomial tuple $\underline{\beta}_f$ of $f$ are of the form $b^ja^i$, for some $i \in \{0,1 \}$ and $j \in \{0,1, \ldots, 2k-1\}$, where
\[
a=(c_0,c_1)(c_2,c_3)\cdots(c_{4k-2},c_{4k-1}),
\]
and
\begin{equation*}
    \begin{split}
        b=& (c_0, c_{2k-2}, \ldots, c_{2m}, c_{2(2m+k-1)}, \ldots , c_{k-2}, c_{4k-6},c_k,c_{2k}, \ldots, c_{2n+k},c_{2(2n+k)}, \ldots, c_{2k-4}, c_{4k-8},\\& {\bf c_{4k-1}, c_{4k-4}}) (c_1, c_{2k+1}, \ldots, c_{2m+1}, c_{2(2m+k)+1}, \ldots , c_{k-1}, c_{4k-3},c_{k+1},c_{2k-1}, \ldots, c_{2n+k+1},\\&c_{2(2n+k)-1}, \ldots, c_{2k-3}, c_{4k-9}, {\bf c_{4k-2}, c_{4k-5}}).
    \end{split}
\end{equation*}
Without loss of generality, we can assume that $\underline{\beta}_f=(I,b^{1},\ldots,b^{2k-1},a,b^{1}a,\ldots,b^{2k-1}a)$. Let $h: \F_q \rightarrow \F_q$ be a function defined as follows:
\[
    \begin{cases}
       h(c_{2i})= b^{4i}(c_{2i}) & \text{ if } 0 \leq i \leq \frac{k-2}{2}, \\
       h(c_{2i+1})= b^{4i+1}(c_{2i+1}) & \text{ if } 0 \leq i \leq \frac{k-2}{2}, \\
       h(c_{2i})= b^{4i-(2k-1)}a(c_{2i}) & \text{ if } \frac{k}{2}  \leq i \leq k-2, \\
       h(c_{2i+1})= b^{4i-2k}a(c_{2i+1}) & \text{ if } \frac{k}{2}  \leq i \leq k-2, \\
       h(c_{4i-2})= b^{4i-(2k-2)}(c_{4i-2}) & \text{ if } \frac{k}{2}  \leq i \leq k-1, \\
        h(c_{4i-1})= b^{4i-(2k-3)}(c_{4i-1}) & \text{ if } \frac{k}{2}  \leq i \leq k-1, \\
        h(c_{4i})= b^{4i-(2k-3)}a(c_{4i}) & \text{ if } \frac{k}{2}  \leq i \leq k-1, \\
        h(c_{4i+1})= b^{4i-(2k-2)}a(c_{4i+1}) & \text{ if } \frac{k}{2}  \leq i \leq k-1, \\
        h(c_{4k-1})= b^{2k-3}a(c_{4k-1}),\\
         h(c_{4k-2})= b^{2k-4}a(c_{4k-2}).\\
	\end{cases}
\]
Our aim is to show that $h$ is a permutation and it intersects the permutation $b^va^u$ simply, for all $u \in \{0,1\}$ and $v \in \{0,1, \ldots, 2k-1\}$. From the property of the permutation polynomial tuple $\underline{\beta}_f$, we know that for any $u_1,u_2 \in \{0,1\}$ and $v_1,v_2 \in \{0,1, \ldots, 2k-1\}$, $(u_1,u_2,v_1,v_2) \neq (0,0,0,0)$, if $b^{v_1}a^{u_1}(c_i)=b^{v_2}a^{u_2}(c_i)=h(c_i)$ for some $i \in \{0,1,\ldots, 4k-1 \}$ then the first equality would imply that $(u_1,v_1)= (u_2, v_2)$. Also, it is easy to see from the definition of the function $h$ that for any fixed  $u \in \{0,1\}$ and $v \in \{0,1, \ldots, 2k-1\}$ there exists a unique $i \in \{0,1,\ldots, 4k-1 \}$ such that $b^va^u(c_i)=h(c_i)$. Thus, it only remains to show that $h$ is a permutation. By definition of $h$, we have
\[
    \begin{cases}
       h(c_{2i})= b^{4i}(c_{2i}) & \text{ if } 0 \leq i \leq \frac{k-2}{2}, \\
       h(c_{2i+1})= b^{4i+1}(c_{2i+1}) & \text{ if } 0 \leq i \leq \frac{k-2}{2}, \\
       h(c_{2i})= b^{4i-(2k-1)}(c_{2i+1}) & \text{ if } \frac{k}{2}  \leq i \leq k-2, \\
       h(c_{2i+1})= b^{4i-2k}(c_{2i}) & \text{ if } \frac{k}{2}  \leq i \leq k-2, \\
       h(c_{4i-2})= b^{4i-(2k-2)}(c_{4i-2}) & \text{ if } \frac{k}{2}  \leq i \leq k-1, \\
        h(c_{4i-1})= b^{4i-(2k-3)}(c_{4i-1}) & \text{ if } \frac{k}{2}  \leq i \leq k-1, \\
        h(c_{4i})= b^{4i-(2k-3)}(c_{4i+1}) & \text{ if } \frac{k}{2}  \leq i \leq k-1, \\
        h(c_{4i+1})= b^{4i-(2k-2)}(c_{4i}) & \text{ if } \frac{k}{2}  \leq i \leq k-1, \\
        h(c_{4k-1})= b^{2k-3}(c_{4k-2}),\\
         h(c_{4k-2})= b^{2k-4}(c_{4k-1}).\\
	\end{cases}
\]

Therefore, we can write $h(\F_q)=C \cup D$, where $C=C_1 \cup C_2 \cup C_3 \cup C_4, D=D_1 \cup D_2 \cup D_3 \cup D_4$,  

\begin{equation*}
 \begin{split}
C_1&=\left \{b^{4i}(c_{2i}) \mid 0 \leq i \leq \frac{k-2}{2}\right \}=\left \{c_0,b^{4}(c_{2}),\ldots,b^{2k-4}(c_{k-2})\right \},\\
C_2&=\left \{b^{4i-2k}(c_{2i}) \mid  \frac{k}{2}  \leq i \leq k-2\right \} \cup \left \{b^{2k-4}(c_{4k-1})\right \}\\&=\left \{c_k,b^{4}
(c_{k+2}),\ldots,b^{2k-8}(c_{2k-4}),b^{2k-4}(c_{4k-1})\right\},\\
C_3&=\left \{b^{4i-(2k-2)}(c_{4i-2}) \mid \frac{k}{2}  \leq i \leq k-1 \right \}=\left \{b^{2}(c_{2k-2}),b^{6}(c_{2k+2}), \ldots, b^{2k-2}(c_{4k-6})\right\},
\end{split}
\end{equation*}

\begin{equation*}
 \begin{split}
C_4&=\left \{b^{4i-(2k-2)}(c_{4i}) \mid \frac{k}{2}  \leq i \leq k-1 \right \}=\left \{b^{2}(c_{2k}),b^{6}(c_{2k+4}), \ldots, b^{2k-2}(c_{4k-4})\right\},\\
D_1&=\left \{b^{4i+1}(c_{2i+1}) \mid 0 \leq i \leq \frac{k-2}{2}\right \}=\left \{b(c_1),b^{5}(c_{3}),\ldots,b^{2k-3}(c_{k-1})\right \},\\
D_2&=\left \{b^{4i-2k+1}(c_{2i+1}) \mid  \frac{k}{2}  \leq i \leq k-2\right \} \cup \left \{b^{2k-3}(c_{4k-2})\right \}\\&=\left \{b(c_{k+1}),b^{5}
(c_{k+3}),\ldots,b^{2k-7}(c_{2k-3}), b^{2k-3}(c_{4k-2})\right\},\\
D_3&=\left \{b^{4i-(2k-3)}(c_{4i-1}) \mid \frac{k}{2}  \leq i \leq k-1 \right \}=\left \{b^{3}(c_{2k-1}),b^{7}(c_{2k+3}), \ldots, b^{2k-1}(c_{4k-5})\right\}, \text{ and}\\
D_4&=\left \{b^{4i-(2k-3)}(c_{4i+1}) \mid \frac{k}{2}  \leq i \leq k-1 \right \}=\left \{b^{3}(c_{2k+1}),b^{7}(c_{2k+5}), \ldots, b^{2k-1}(c_{4k-3})\right\}.
\end{split}
\end{equation*}
It can be observed that the images of $h$ are of the form $b^{t_1}(c_i)$ for some  $t_1 \in \{0,1,\ldots,2k-1\}$ and $i \in \{0,1,\ldots,q-1\}$. So, it is easy to see that $C \cap D=\phi$ as input for $b^{t_1}$ in the sets $C$ and $D$ are from two different disjoint cycles of $b$. Thus, in order to show that $h$ is a permutation of $\F_q$, equivalently, $C \cup D= \F_q$, it only remains to show that all the elements of the sets $C $ and $ D$ are distinct which follows from Lemma \ref{L45}. Now using Lemma \ref{T43}, the bivariate LPP $g$ corresponding to the permutation polynomial tuple $\underline{\beta}_g=(h,hb^{1},\ldots,hb^{2k-1},ha,hb^{1}a,\ldots,hb^{2k-1}a)$ is a companion of $f$.
\end{proof}
The following theorem gives a companion for the permutation group polynomials proposed in Theorem \ref{T32}. 
 \begin{thm}\label{T47}
 Let $q=2t$, where $t=2^m$ and $m \geq 2$ is a positive integer and  $f\in \F_q[X,Y]$ be a permutation group polynomial given in Theorem \ref{T32}. Then $f$ has a  companion.
 \end{thm}
 \begin{proof} Let $f \in \F_q[X,Y]$ be a permutation group polynomial in  Theorem \ref{T32} and $\underline{\beta}_f$ be its corresponding permutation polynomial tuple. Without loss of generality,  we  can  write  $\underline{\beta}_f$ of the form $(I,b^{1},\ldots,b^{t-1},a,b^{1}a,\ldots,b^{t-1}a)$, where
\[
a=(c_0,c_1)(c_2,c_3)\cdots(c_{2t-2},c_{2t-1}),
\]
and
\begin{equation*}
    \begin{split}
        b= &(c_{t-2},\ldots,c_{t-2(m+1)}, \ldots,c_{0},  c_t,\ldots, c_{t+2m},\ldots,c_{2t-2})\\
        &(c_{1},\ldots, c_{2m+1}, \ldots,c_{t-1},c_{2t-1},\ldots, c_{2t-(2m+1)}, \ldots,c_{t+1}).
    \end{split}
\end{equation*}

Now, we shall construct a permutation polynomial  $h$ from $\F_q$ to itself such that it intersects the permutation $b^va^u$ simply, for all $u \in \{0,1\}$ and $v \in \{0,1, \ldots, t-1\}$. Then using Lemma \ref{T43}, the bivariate LPP $g$ corresponding to the permutation polynomial tuple $\underline{\gamma}_g=(h,hb^{1},\ldots,hb^{t-1},ha,hb^{1}a,\ldots,hb^{t-1}a)$ will be a companion of $f$.  Let $h: \F_q \rightarrow \F_q$ be defined as follows:
\[
h(c_{i})=
\begin{cases}
	b^{i}(c_{i})~&~\mbox{if}~0 \leq i \leq t-1,\\
    b^{2t-i-1}a(c_{i})~&~\mbox{if}~t \leq i \leq 2t-1.
\end{cases}
\]
Thus, it only remains to show  that $h$ is a permutation of $\F_q$ and it will  intersect each permutation $b^va^u$ simply, where $u \in \{0,1\}$ and $v \in \{0,1, \ldots, t-1\}$. Since $\underline{\beta}_f$ is a permutation polynomial tuple, therefore for any $u_1,u_2 \in \{0,1\}$ and $v_1,v_2 \in \{0,1, \ldots, t-1\}$, $(u_1,u_2,v_1,v_2) \neq (0,0,0,0)$, if  $b^{v_1}a^{u_1}(c_i)=b^{v_2}a^{u_2}(c_i)=h(c_i)$ for some $i \in \{0,1,\ldots, 2t-1 \}$ then we have $(u_1,v_1)= (u_2, v_2)$. Also, it is easy to see from the definition of the function $h$ that for any fixed  $u \in \{0,1\}$ and $v \in \{0,1, \ldots, t-1\}$ there exists a unique $i \in \{0,1,\ldots, 2t-1 \}$ such that $b^va^u(c_i)=h(c_i)$. Next, we shall prove that $h$ is a permutation of $\F_q$. We observe that
\[
\begin{cases}
	h(c_{2j})= b^{2j}(c_{2j})~&~\mbox{if}~0\leq  j \leq \frac{t}{2}-1,\\
    h(c_{2j})= b^{2t-2j-1}(c_{2j+1})~&~\mbox{if}~ \frac{t}{2} \leq j \leq t-1,\\
    h(c_{1+2j})= b^{1+2j}(c_{1+2j})~&~\mbox{if}~0\leq j \leq \frac{t}{2}-1,\\
    h(c_{1+2j})= b^{2t-2j-2}(c_{2j})~&~\mbox{if}~\frac{t}{2} \leq j \leq t-1.
\end{cases}
\]
Thus, $h(\F_q)=C_1 \cup C_2 \cup D_1 \cup D_2$, where
\begin{equation*}
    \begin{split}
        C_1 &= \{c_0, b^2(c_2), \ldots, b^{t-2}(c_{t-2})\},\\
        C_2 &= \{b(c_1), b^3(c_3), \ldots, b^{t-1}(c_{t-1})\},\\
        D_1 &= \{b^{t-1}(c_{t+1}), b^{t-3}(c_{t+3}), \ldots, b(c_{2t-1})\},\\
        D_2 &= \{b^{t-2}(c_t), b^{t-4}(c_{t+2}), \ldots, c_{2t-2}\}.\\
    \end{split}
\end{equation*}
It is easy to see that $C_1 \cap C_2 = \phi$ and $D_1 \cap D_2 = \phi$. Similarly, $C_1 \cap D_1 = \phi$ and $C_2 \cap D_2 = \phi$. Therefore, in order to show that $C_1 \cup C_2 \cup D_1 \cup D_2= \F_q$, it is enough to prove  that all the elements of the sets $C_1 \cup D_2$ and $C_2 \cup D_1$ are distinct, which follows from Lemma~\ref{L45}.
\end{proof}

The following theorem provides a companion for permutation group polynomials obtained in  Theorem \ref{T33}.

\begin{thm}\label{T48}
 Let $q=4k$, where $k=2^{\ell}$ for some integer $\ell  \geq 2$ and  $f\in \F_q[X,Y]$ be an permutation group polynomial given in Theorem \ref{T33}. Then $f$ has a companion.
 \end{thm}
 \begin{proof}
 Let $\underline{\beta}_f$ be the corresponding permutation polynomial tuple of the permutation group polynomial $f$ in Theorem \ref{T33}. Without loss of generality, we can express $\underline{\beta}_f$ as follows:
 \[
 \underline{\beta}_f=(I,b^{1},\ldots,b^{t-1},a,b^{1}a,\ldots,b^{t-1}a), 
 \]
 where 
\[
a=(c_0,c_1)(c_2,c_3)\cdots(c_{4k-2},c_{4k-1}),
\]
and
\begin{equation*}
    \begin{split}
        b=& (c_0, c_{2k+1}, \ldots, c_{2i}, c_{2i+2k+1}, \ldots , c_{k-2}, c_{3k-1},{\bf c_{k+1},c_{3k}},c_{k+2},c_{3k+2}, \ldots, c_{2(j+1)+k},c_{2(j+1)+3k}, \\& \ldots, c_{2k-2}, c_{4k-2}) (c_{3k-2},c_{2k-1}, \ldots, c_{3k-2-2j}, c_{2k-1-2j}, \ldots ,c_{2k+2},c_{k+3},{\bf c_{2k},c_{k}, }  c_{4k-1}, c_{k-1},\ldots, \\&, c_{4k-1-2i},c_{k-1-2i}, \ldots, c_{3k+1},c_{1}). 
    \end{split}
\end{equation*}

Now, our aim is to  construct a permutation function $h$ from $\F_q$ to itself which intersects $\underline{\beta}_f$ simply. Then from Lemma \ref{T43}, $g$ is a companion of $f$, where $g$ is the bivariate LPP associated to the permutation polynomial tuple $(h,hb^{1},\ldots,hb^{t-1},ha,hb^{1}a,\ldots,hb^{t-1}a)$. Next,  let $h: \F_q \rightarrow \F_q$ be a function defined as follows:

\[
    \begin{cases}
       h(c_{2i})= b^{4i}(c_{2i})& \text{ if } 0 \leq i \leq \frac{k-2}{2}, \\
       h(c_{2i+1})= b^{4i+1}(c_{2i+1})& \text{ if } 0 \leq i \leq \frac{k-2}{2}, \\
       h(c_{2i})= b^{4i-(2k-1)}a(c_{2i})&  \text{ if }  \frac{k+2}{2}  \leq i \leq k-1, \\
       h(c_{2i+1})= b^{4i-2k}a(c_{2i+1})&  \text{ if }  \frac{k+2}{2}  \leq i \leq k-1, \\

       h(c_{2i})= b^{4(i-k)+3}(c_{2i})&  \text{ if }  k  \leq i \leq \frac{3k}{2}-1, \\
        h(c_{2i+1})= b^{4(i-k)+2}(c_{2i+1})&  \text{ if }  k  \leq i \leq \frac{3k}{2}-1, \\
        h(c_{2i})= b^{4(i-\frac{3k}{2})+3}a(c_{2i})&  \text{ if }  \frac{3k}{2}  \leq i \leq 2k-1, \\
        h(c_{2i+1})= b^{4(i-\frac{3k}{2})+2}a(c_{2i+1})&  \text{ if }  \frac{3k}{2}  \leq i \leq 2k-1, \\ 
        h(c_k)= b^{0}a(c_k),\\
         h(c_{k+1})= b^{1}a(c_{k+1}).\\
	\end{cases}
\]
 Therefore, from the definition of permutation polynomial tuple and the definition of the function $h$,  we have that for any fixed  $u \in \{0,1\}$ and $v \in \{0,1, \ldots, t-1\}$ there exists a unique and distinct $i \in \{0,1,\ldots, 4k-1 \}$ such that $b^va^u(c_i)=h(c_i)$. Thus, it only remains to prove that $h$ is a permutation of $\F_q$.  Notice that the image set of $h$ can be written as follows:
 \[
h(\F_q)=
    \begin{cases}
       h(c_{2i})= b^{4i}(c_{2i}) &  \text{ if }  0 \leq i \leq \frac{k-2}{2}, \\
       h(c_{2i+1})= b^{4i+1}(c_{2i+1})&   \text{ if }  0 \leq i \leq \frac{k-2}{2}, \\
       h(c_{2i})= b^{4i-(2k-1)}(c_{2i+1})&  \text{ if }  \frac{k+2}{2}  \leq i \leq k-1, \\
       h(c_{2i+1})= b^{4i-2k}(c_{2i}) &  \text{ if }  \frac{k+2}{2}  \leq i \leq k-1, \\

       h(c_{2i})= b^{4(i-k)+3}(c_{2i}) &  \text{ if }  k  \leq i \leq \frac{3k}{2}-1, \\
        h(c_{2i+1})= b^{4(i-k)+2}(c_{2i+1}) &  \text{ if }  k  \leq i \leq \frac{3k}{2}-1, \\
        h(c_{2i})= b^{4(i-\frac{3k}{2})+3}(c_{2i+1}) &  \text{ if }  \frac{3k}{2}  \leq i \leq 2k-1, \\
        h(c_{2i+1})= b^{4(i-\frac{3k}{2})+2}(c_{2i})&  \text{ if }  \frac{3k}{2}  \leq i \leq 2k-1, \\ 
        h(c_k)= b^{0}(c_{k+1}),\\
         h(c_{k+1})= b^{1}(c_{k}).\\
	\end{cases}
\]
So, $h(\F_q)=C \cup D$, where $C=C_1 \cup C_2 \cup C_3 \cup C_4, D=D_1 \cup D_2 \cup D_3 \cup D_4$, 

\begin{equation*}
 \begin{split}
C_1&=\left \{b^{4i}(c_{2i}) \mid 0 \leq i \leq \frac{k-2}{2}\right \}=\left \{c_0,b^{4}(c_{2}),\ldots,b^{2k-4}(c_{k-2})\right \},\\
C_2&=\left \{ b^{4i-2k}(c_{2i}) \mid \frac{k+2}{2}  \leq i \leq k-1\right \} \cup \left \{b^{0}(c_{k+1})\right \}=\left \{c_{k+1},b^{4}
(c_{k+2}),\ldots,b^{2k-4}(c_{2k-2})\right\},\\
C_3&=\left \{b^{4(i-k)+2}(c_{2i+1}) \mid  k  \leq i \leq \frac{3k}{2}-1 \right \}=\left \{b^{2}(c_{2k+1}),b^{6}(c_{2k+3}), \ldots, b^{2k-2}(c_{3k-1})\right\},\\
C_4&=\left \{b^{4(i-\frac{3k}{2})+2}(c_{2i}) \mid  \frac{3k}{2}  \leq i \leq 2k-1 \right \}=\left \{b^{2}(c_{3k}),b^{6}(c_{3k+2}), \ldots, b^{2k-2}(c_{4k-2})\right\} \\
D_1&=\left \{b^{4i+1}(c_{2i+1}) \mid 0 \leq i \leq \frac{k-2}{2}\right \}=\left \{b(c_1),b^{5}(c_{3}),\ldots, b^{2k-3}(c_{k-1})\right \},\\
D_2&=\left \{  b^{4i-(2k-1)}(c_{2i+1}) \mid  \frac{k+2}{2}  \leq i \leq k-1 \right \} \cup \left \{b^{1}(c_{k})\right \} \\
&=\left \{b(c_{k}),b^{5} (c_{k+3}),\ldots,b^{2k-7}(c_{2k-3}),b^{2k-3}(c_{2k-1})\right\},\\
D_3&=\left \{b^{4(i-\frac{3k}{2})+3}(c_{2i+1}) \mid \frac{3k}{2}  \leq i \leq 2k-1 \right \}=\left \{b^{3}(c_{3k+1}),b^{7}(c_{3k+3}), \ldots, b^{2k-1}(c_{4k-1})\right\}, \text{ and}\\
D_4&=\left \{ b^{4(i-k)+3}(c_{2i}) \mid k  \leq i \leq \frac{3k}{2}-1 \right \}=\left \{b^{3}(c_{2k}),b^{7}(c_{2k+2}), \ldots, b^{2k-1}(c_{3k-2})\right\}.
\end{split}
\end{equation*}
Now, one can  see that every element of $h(\F_q)$ can be written as $b^{t_1}(c_i)$ for some  $t_1 \in \{0,1,\ldots,2k-1\}$ and $i \in \{0,1,\ldots,q-1\}$. Using the similar argument as in Theorem~\ref{T46} and using Lemma \ref{L45}, we observe that $C \cap D=\phi$ and $h(\F_q)=C \cup D=\F_q$.
\end{proof}
\section{Conclusion}\label{S5}
We have constructed three families of bivariate permutation group polynomials over finite fields with even characteristic. These permutation group polynomials are inequivalent both to the known family of permutation group polynomials and to each other. We also investigated the existence of companions for $e$- Klenian polynomials over finite fields with even characteristic, building on a recent result by Gutierrez and Urroz, who demonstrated the existence of a companion for $e$- Klenian polynomials when the characteristic of the finite field is odd. Moreover, we provided companions for each of the new families of LPPs we introduced. These companions can be used to construct complete set of MOLS. It would be an interesting problem to determine whether these complete sets of MOLS are distinct from those known in the literature.
\section*{Acknowledgements}
We would like to sincerely thank Mohit Pal for his careful reading of the initial draft and for his many valuable discussions.

\end{document}